\documentclass[11pt]{amsart}
\usepackage{amsfonts}
\usepackage{graphicx,color}
\usepackage{amsthm}
\usepackage{amssymb,amsmath}

\title[Upper energy bounds for spherical codes]{\bf Upper bounds for energies of spherical codes of given cardinality and separation}

\date{\today}

\newtheorem{theorem}{Theorem}[section]

\newtheorem{proposition}[theorem]{Proposition}

\theoremstyle{definition}
\newtheorem{definition}[theorem]{Definition}

\newtheorem{remark}[theorem]{Remark}

\newcommand{\R}{\mathbb{R}}

\author[P. Boyvalenkov]{P. G. Boyvalenkov$^\dagger$}
\address{Institute of Mathematics and Informatics, Bulgarian Academy of Sciences,
8 G Bonchev Str.,
1113  Sofia, Bulgaria \\
and Technical Faculty, South-Western University, Blagoevgrad, Bulgaria.
}
\email{peter@math.bas.bg}
\thanks{\noindent $^\dagger$ The research of this author was partially supported by the National Scientific Program "Information and Communication Technologies for a Single Digital Market in Science, Education and Security (ICTinSES)", financed by the Bulgarian Ministry of Education and Science.}

\author[P. Dragnev]{P. D. Dragnev $^{\dagger \dagger}$}
\address{Department of Mathematical Sciences,
Purdue University Fort Wayne, IN 46805, USA }
\email{dragnevp@pfw.edu}
\thanks{\noindent $^{\dagger \dagger}$ The research of this author was supported, in part, by a Simons Foundation grant no. 282207.}

\author[D. Hardin]{D. P. Hardin$^*$}
\address{Center for Constructive Approximation, Department of Mathematics, \hspace*{.1in}
Vanderbilt University,
Nashville, TN 37240, USA  }
\email{doug.hardin@vanderbilt.edu}

\author[E. Saff]{E. B. Saff$^*$}
\email{edward.b.saff@vanderbilt.edu}
\thanks{\noindent $^*$ The research of these authors was supported, in part,
by the U. S. National Science Foundation under grant  DMS-1516400.
}
\author[M. Stoyanova]{M. M. Stoyanova$^{\dagger \dagger \dagger}$}
\address{Faculty of Mathematics and Informatics,
Sofia University,
5 James Bourchier Blvd.,
1164 Sofia, Bulgaria}
\email{stoyanova@fmi.uni-sofia.bg}
\thanks{\noindent $^{\dagger \dagger \dagger}$ The research of this author was supported by a Bulgarian NSF contract DN2/02-2016. }
\thanks{Research for this article was started while the authors were in residence at the Institute for Computational and Experimental Research in Mathematics in Providence, RI, during the ``Point Configurations in Geometry, Physics and Computer Science" program supported by the National Science Foundation under Grant No. DMS-1439786.}

\begin{document}
\maketitle

\begin{abstract}
We introduce a linear programming framework for obtaining upper bounds for the potential energy of spherical codes of
fixed cardinality and minimum distance. Using Hermite interpolation we construct polynomials to derive
corresponding bounds. These bounds are universal in the sense that they are valid for all
absolutely monotone potential functions and the required interpolation nodes do not depend on the potentials.
\end{abstract}



\section{Introduction}

Let $\mathbb{S}^{n-1}$ denote the unit sphere in $\mathbb{R}^n$ and $C \subset \mathbb{S}^{n-1}$ be a {\em spherical code}; i.e. a finite subset of $\mathbb{S}^{n-1}$.
Given an (extended real-valued) function $h:[-1,1] \to [0,+\infty]$, the (unnormalized) {\em potential energy} (or {\em $h$-energy})
of $C$ is given by
\begin{equation}
\label{CodeEnergy}
E_{h}(C):=\sum_{x, y \in C, x \neq y} h(\langle x,y \rangle), \end{equation}
where $\langle x,y \rangle$ denotes the usual inner product of $x$ and $y$.

Denote  by
\[ s(C):=\max\{\langle x,y \rangle: x,y \in C, x \neq y\}, \]
the maximal inner product of a spherical code $C$ and by
\[ C(n,M,s):=\{ C \subset \mathbb{S}^{n-1}: |C|=M, s(C)=s\} \]
the family of all spherical codes on $\mathbb{S}^{n-1}$ of given cardinality $M$ with maximal inner product $s$.
Note that the set $C(n,M,s)$ can be empty, for example $C(n,n+1,s)$ is empty for $s<-1/n)$ and $C(n,2n,s)$ is empty for $s<0$.

Given, $n$, $M$, $s$, and $h$, we are interested in upper bounds on the quantity
\begin{equation}
 \mathcal{G}_h(n,M,s):=\sup_{C \in C(n,M,s)} \{E_h(C) \},
\end{equation}
where we use the convention that supremum of the empty set is $-\infty$.
Hereafter, we shall consider the class $\mathcal{AM}([-1,1])$ of potentials $h$, which are \emph{absolutely monotone in} $[-1,1]$; that is,
extended real-valued functions $h:[-1,1] \to (0,+\infty]$ such that
$h^{(k)}(t)\geq 0$ for every $t\in[-1,1)$ and every integer $k \geq 0$, where $h(1):=\lim_{t \to 1^{-}} h(t)$. Among the most prominent absolutely monotone potentials we list
\[ h(t)=[2(1-t)]^{1-n/2}, \mbox{ Newton potential},\]
\[ h(t)=[2(1-t)]^{-\alpha/2}, \ \alpha>0, \mbox{ Riesz potential}, \]
\[ h(t)=e^{-\alpha (1-t)}, \mbox{ Gaussian potential}, \]
\[ h(t)=-\log [2(1-t)], \mbox{ Logarithmic potential}. \]

Many important potential interactions tend to infinity when $t$ tends to $1^{-}$. Therefore, for obtaining finite upper energy bounds, it is necessary to impose restrictions on the separation $s$ since as $s$ tends to $1^-$ the energy of the code tends to infinity. 

One natural restriction is to consider upper energy bounds for codes that minimize $h$-energy for specified or general $h;$ 
such upper bounds on the minimal energy for Riesz ($0<\alpha<2$) and logarithmic potentials on $\mathbb{S}^2$ have 
been considered by Wagner in \cite{W}.

Another restriction leads to the class of spherical $\tau$-designs. Energy bounds for the Coulomb energy of
spherical designs on $\mathbb{S}^2$ were considered by Hesse and Leopardi \cite{HL08} (see also \cite{Hes09}).
Recently, more general results (including lower and upper bounds) for Riesz and logarithmic energy of
spherical designs with relatively small cardinalities were obtained by Grabner and Stepanyuk \cite{GS19}
(see also \cite{Ste19}). Universal upper and lower energy bounds for spherical designs
 were obtained by the present authors for all absolutely monotone potentials in \cite{BDHSS-drna}. In all these cases the 
spherical designs considered are well-separated (the asymptotic existence of such designs was proved by Bondarenko-Radchenko-Viazovska \cite{BRV15}).  We also remark that Leopardi \cite{Leo13} examined bounds on the normalized Riesz energy for a class of spherical codes that is both well-separated and asymptotically (as $M\to \infty$) equidistributed.   

In this paper we consider yet another possibility -- spherical codes (not necessarily spherical designs of large strength) 
with prescribed cardinality and maximal inner product (equivalently, minimum distance or separation). We derive a general 
linear programming approach in the spirit of Delsarte-Yudin for obtaining upper bounds on $\mathcal{G}_h(n,M,s)$ along with conditions under which these bounds are sharp (see Theorem~\ref{thm construction}). 
The effectiveness of these bounds relies on the construction of suitable upper bounding polynomials that we show to 
exist for all potentials from $\mathcal{AM}([-1,1])$.  We then test our estimates on some relevant codes.

Recently, we developed \cite{BDHSS-drna,BDHSS-ca,BDHSS-dcc19,BDHSS-subm,BDHSS-hd} linear programming techniques for
obtaining universal lower bounds for energy of spherical codes of different classes\footnote{Similar results for codes in Hamming 
spaces are obtained in \cite{BDHSS-dcc17,BDHSS-subm,BDHSS-subm19}.}. With the addition of the upper bounds from this paper we
establish an "energy strip" where all energies of codes from $C(n,M,s)$ belong.


The paper is organized as follows. In Section 2 we present a general linear programming upper bound on $\mathcal{G}_h(n,M,s),$ see Theorem 2.2. In  Section 3, with the help of Levenshtein-type polynomials, we construct `feasible' polynomials that provide the energy bound for
absolutely monotone potentials. The main result is Theorem \ref{thm construction} where a new universal upper bound for
$\mathcal{G}_h(n,M,s)$ is asserted. The bound is universal in the sense that it is a linear combination with positive weights of the values of the potential function on a collection of nodes, where both weights and nodes are independent of the potential function.
Theorem \ref{test-f-th} provides a necessary condition for optimality of the bound from
Theorem \ref{thm construction}. Examples and further discussions are provided in Section 4.

\section{Linear programming for upper bounds for  $\mathcal{G}_h(n,M,s)$}

Let $P_i^{(n)}(t)$, $i=0,1,\ldots$, be the Gegenbauer polynomials normalized by  $P_i^{(n)}(1)=1$, which satisfy the following three-term recurrence relation
\[ (i+n-2)\, P_{i+1}^{(n)}(t)=(2i+n-2)\, t\, P_i^{(n)}(t)-i\, P_{i-1}^{(n)}(t)
                \mbox{ for } i \geq 1, \]
where $P_0^{(n)}(t):=1$ and $P_1^{(n)}(t):=t$.  In standard Jacobi polynomial notation (see \cite[Chapter 4]{Sze}), we have that
\begin{equation}\label{Geg_Jacobi}
P_i^{(n)}(t)=\frac{P_i^{((n-3)/2,   (n-3)/2)}(t)}{P_i^{((n-3)/2,   (n-3)/2)}(1)}.
\end{equation}

If $f$ is a continuous function in $[-1,1]$, then $f$ can be uniquely expanded in terms of the Gegenbauer polynomials as
\begin{equation}\label{Exp-Geg}
f(t) = \sum_{i=0}^{\infty} f_iP_i^{(n)}(t),
\end{equation}
where the convergence is in $L_2([-1,1])$ and the coefficients $f_i$ are given by
\begin{equation}
f_i=\gamma_n \int_{-1}^1f(t)P_i^{(n)}(t)(1-t^2)^{(n-3)/2} \, dt
\end{equation}
where
\[ \gamma_n=1/\int_{-1}^1 \left[P_i^{(n)}(t)\right]^2(1-t^2)^{(n-3)/2}\, dt = \frac{\Gamma (n/2)}{\sqrt{\pi } \Gamma ((n-1)/2)} \]
is a normalizing constant. If the coefficients $f_i$ are eventually all of the same sign, then it is a classical result of Schoenberg  \cite{Sch1942} that the series on the right-hand side of  \eqref{Exp-Geg} converges uniformly and absolutely to $f$.

\begin{definition} \label{FeasibleSet} For fixed $n \geq 2$, $s \in [-1,1)$, and given $h,$ denote by $U_h^{n,s}$ the {\em feasible set} of functions $f\in C([-1,1]);$ that is, the functions satisfying
\begin{itemize}
\item[{\rm (F1)}] $f(t) \geq h(t)$ for every $t \in [-1,s]$ and
\vskip 1mm
\item[{\rm (F2)}] the coefficients in the Gegenbauer expansion \eqref{Exp-Geg} satisfy $f_i \leq 0$ for $i= 1,2,3,\ldots$.
\end{itemize}
\end{definition}

For a spherical code $C \subset \mathbb{S}^{n-1}$
and a postive integer $i$, the {\em $i$-th moment} of $C$ is defined by
 \begin{equation}\label{Mk0}
 M_i(C):=\sum_{x,y\in C} P^{(n)}_i(\langle x , y \rangle ).\end{equation}
The well known positive definiteness of the Gegenbauer polynomials implies that $M_i(C) \geq 0$ for every nonnegative integer $i$.
If $M_i(C)=0$ for every $i \in \{1,2,\ldots,\tau\}$, then  $C$  is called a {\em spherical $\tau$-design}.
Spherical designs were introduced in 1977 by Delsarte, Goethals and Seidel in their seminal paper \cite{DGS77} with
the following equivalent definition (among others): $C$ is  a spherical $\tau$-design if and only if \begin{eqnarray*}
\label{defin_f.1}
 \int_{\mathbb{S}^{n-1}} p(x) d\sigma_n(x)= \frac{1}{|C|} \sum_{x \in C} p(x)
\end{eqnarray*}
 holds for all polynomials
$p(x) = p(x_1,x_2,\ldots,x_n)$ of total degree at most $\tau$ where  $\sigma_n $ denotes the normalized $(n-1)$-dimensional Hausdorff measure. The largest $\tau$ such that $C$ is a spherical $\tau$-design is called the \emph{strength} of $C$.

The following Delsarte-Yudin type linear programming theorem is a key tool for obtaining the upper bounds  for the quantity $\mathcal{G}_h(n,M,s)$ given in Theorem~\ref{thm construction}.

\begin{theorem}\label{thm 1}
Let $n \geq 2$, $M \geq 2$ be positive integers, $s \in [-1,1)$, and $h:[-1,1)\to\R$.
If $f \in U_h^{n,s}$ and   $C\in C(n,M,s)$, then
\begin{equation}\label{DYbound} E_h(C)\le \mathcal{G}_h(n,M,s) \leq M(f_0M-f(1)),\end{equation}
 with equality holding throughout \eqref{DYbound} if and only if both conditions

 (a) $f(t)=h(t)$ for every $t\in \{\langle x,y\rangle: x\neq y\in C\}$;

(b) $f_iM_i(C)=0$ for all $i=1,2,3,\ldots$

\noindent hold. 
\end{theorem}

\begin{proof}
Although this result can be deduced from the lower bound for energy given in \cite{Y} (see also \cite[Theorem 5.5.1]{BHS20}), we include a direct proof here for the convenience of the reader.

Let  $C \subset \mathbb{S}^{n-1}$ be a spherical code.  Since $f(t) \in U_h^{n,s}$, its Gegenbauer expansion   $f(t)= \sum_{i=0}^{\infty} f_i P_i^{(n)}(t)$ converges
uniformly on $[-1,1]$ to $f$.  Thus,
\begin{equation}
\label{main-identity}
f(1)|C|+\sum_{x \in C} \sum_{y \in C \setminus \{x\}} f(\langle x,y \rangle )=f_0|C|^2+\sum_{i>0} f_iM_i(C),
\end{equation}
where the right-hand side of
\eqref{main-identity} is obtained using \eqref{Mk0} and interchanging the order of summation.

If $C \in C(n,M,s)$ and $f \in U_h^{n,s}$, then the condition (F1) together with $s(C)=s$ imply that the left hand side of \eqref{main-identity} is at least $Mf(1)+E_h(C)$.
Furthermore, (F2) and the inequalities $M_i(C) \geq 0$ for $i=1,2,\ldots,\deg(f)$
yield that the right-hand side is at most $M^2f_0$. Therefore
\begin{equation}\label{EhCequal} E_h(C)\leq E_f(C)   \leq M(f_0M-f(1)).
\end{equation}
Since these estimations are valid
for every code $C \in C(n,M,s)$ we conclude that the desired bound follows.

Note that $E_h(C) = E_f(C)$ if and only if condition (a) is satisfied, while it follows from \eqref{main-identity} that $E_f(C)   = M(f_0M-f(1))$ if and only if condition (b) is satisfied.
Hence, equality holds in \eqref{EhCequal} if and only if both (a) and (b) are satisfied.
\end{proof}

In the next section, we will construct   polynomials in $U_h^{n,s}$   for any fixed $n$, $s$, and $h \in \mathcal{AM}([-1,1])$ that can be used in conjunction with Theorem \ref{thm 1} to provide explicit upper bounds for $\mathcal{G}_h(n,M,s)$, where $M$ is chosen in accordance with $s$ and $n$.


\section{Construction
of feasible polynomials for Theorem \ref{thm 1}}

In this section we develop methods for constructing polynomials in $U_h^{n,s}$ for a given potential $h$ and parameters $n$ and $s$; we shall refer to such polynomials   as {\em feasible polynomials}. These  methods rely on   the Levenshtein framework (reviewed below)  used to obtain  universal bounds on the cardinality of maximal codes with given separation distance and universal lower bounds on potential energy
of codes of given cardinality (see \cite{lev98}, \cite{BDHSS-ca}, and \cite{BDHSS-hd}).

\subsection{Levenshtein framework parameters and ULB spaces}

We first recall the definition of the {\em Levenshtein function} bounding the quantity
$$A(n,s):=\max\{|C| \colon C \subset \mathbb{S}^{n-1}, \langle x,y \rangle \leq s, \,   x\neq y \in C\}$$ that denotes the maximal possible cardinality of a spherical code on $\mathbb{S}^{n-1}$ of prescribed maximal
inner product $s$.

For $a,b \in \{0,1\}$ and $i \geq 1$, let $t_i^{a,b}$   denote the greatest zero of the (adjacent)
Jacobi polynomial $P_i^{(a+\frac{n-3}{2},b+\frac{n-3}{2})}(t)$ and also define $t_0^{1,1}=-1$.   For $ m \in \mathbb{N}$, let   $I_m$ denote the interval
\begin{eqnarray}\label{L_int}
  I_m :=
\left\{
\begin{array}{ll}
    \left [ t_{k-1}^{1,1},t_k^{1,0} \right ], & \mbox{if } m=2k-1, \\[6pt]
    \left [ t_k^{1,0},t_k^{1,1} \right ],      & \mbox{if } m=2k, \\
  \end{array}\right.
\end{eqnarray}
The collection of intervals $\{I_m\}_{m=1}^\infty$ is well defined from the interlacing properties $ t_{k-1}^{1,1}<t_k^{1,0}<t_k^{1,1}$, see \cite[Lemmas 5.29, 5.30]{lev98}. Note also that it partitions  $I=[-1,1)$ into countably many subintervals with non-overlapping interiors.

For every $s \in I_m$, using linear programming
bounds for special polynomials $f_m^{(n,s)}(t)$ of degree $m$ (see \cite[Equations (5.81) and (5.82)]{lev98}), Levenshtein proved that (see \cite[Equation (6.12)]{lev98})
\begin{equation}
\label{L_bnd}
 A(n,s) \leq
\left\{
\begin{array}{ll} \displaystyle
    L_{2k-1}(n,s) := {k+n-3 \choose k-1}
         \left[ \frac{2k+n-3}{n-1} -
          \frac{P_{k-1}^{(n)}(s)-P_k^{(n)}(s)}{(1-s)P_k^{(n)}(s)}
         \right] ,&\mbox{if }s\in I_{2k-1}, \\[8pt]
\displaystyle    L_{2k}(n,s) := {k+n-2 \choose k}
        \left[ \frac{2k+n-1}{n-1} -
           \frac{(1+s)( P_k^{(n)}(s)-P_{k+1}^{(n)}(s))}
    {(1-s)(P_k^{(n)}(s)+P_{k+1}^{(n)}(s))} \right] ,&\mbox{if } s\in I_{2k}.\cr
   \end{array}\right.
\end{equation}
For every fixed dimension $n,$ each bound $L_m (n,s)$ is smooth with respect to $s$. The {\em Levenshtein function} is defined as
\begin{equation}
\label{L_function}
L(n,s) :=
\left\{
\begin{array}{ll}
    L_{2k-1}(n,s),& \mbox{ if } s \in I_{2k-1}, \\[4pt]
   L_{2k}(n,s), & \mbox{ if } s \in I_{2k}. \cr
    \end{array}\right.
\end{equation}
It is a function that is continuous and strictly increasing in $s$, whose values at the endpoints of the intervals $\mathcal{I}_m$ coincide with the Delsarte-Goethals-Seidel numbers $D(n,m)$ (see \cite{DGS77} for the definition).

Next, we introduce the notion of a {\em $1/N$-quadrature rule} over subspaces consisting of polynomials (see \cite{BDHSS-ca}). The classical example of $1/N$-quadrature rule is given by Levenshtein's Theorem 5.39 in \cite{lev98}, where a Gauss-Jacobi quadrature formula is defined (see \cite{lev92} for the origin of this result).

\begin{definition} For fixed dimension $n \geq 2$ and a real number $N \geq 2$, a finite sequence of ordered pairs
$\{(\alpha_i, \rho_i)\}_{i=1}^{k}$, where $\alpha_1 < \alpha_2 <\cdots <\alpha_{k}$ are nodes ($-1 \leq \alpha_1$ and $\alpha_k<1$)
and $\rho_1,\rho_2,\ldots,\rho_k$ are positive weights, forms a {\em $1/N$-quadrature rule, $N>0$,
that is exact  for the subspace  $\Lambda \subset C([-1,1])$} if the quadrature formula
\begin{equation}
\label{defin_f0.1}
f_0=\gamma_n\int_{-1}^1f(t)(1-t^2)^{(n-3)/2}dt= \frac{f(1)}{N}+ \sum_{i=1}^{k} \rho_i f(\alpha_i),
\end{equation}
holds true for all polynomials $f\in \Lambda$.
\end{definition}

In our terminology, given a code $C\in C(n,M,s)$, we associate (uniquely) $m:=m(n,s)$ such that $s\in \mathcal{I}_m$. Then \eqref{defin_f0.1} is a
$1/L_m(n,s)$-quadrature rule exact for the subspace of real polynomials of degree at most $m$.
Hereafter $L_m(n,s)$, $m=1,2,\ldots$, will denote the Levenshtein function \eqref{L_function} on the interval $\mathcal{I}_m$ as defined in \eqref{L_int}.

The nodes in the Levenshtein's $1/L_m(n,s)$-quadrature rule are the roots of certain polynomial
$f_m^{(n,s)}(t)$  of degree $m$ (see \cite[Theorem 5.39]{lev98}), called the {\em Levenshtein polynomial}, that is used for obtaining the bound \eqref{L_bnd}. The explicit form of $f_m^{(n,s)}(t)$ (see \cite[Eqs. (1.35-36)]{lev92} or \cite[Eq. (3.82)]{lev98}) is given by
\begin{equation}\label{fmnsdef} f_m^{(n,s)}(t)=\prod_{\alpha_i \in T} (t-\alpha_i), \end{equation}
where $T$ is the multiset
\begin{equation} \label{MultisetDef} T=\left\{ \begin{array}{ll}
\{\alpha_0,\alpha_0,\alpha_1,\alpha_1,\ldots,\alpha_{k-2}, \alpha_{k-2}, \alpha_{k-1}\} & \mbox{if } m=2k-1 \\[8pt]
 \{\alpha_0=-1,\alpha_1,\alpha_1,\alpha_2,\alpha_2,\ldots,\alpha_{k-1}, \alpha_{k-1}, \alpha_{k}\} & \mbox{if } m=2k
\end{array} \right. \end{equation}
of cardinality $m$. The reader may easily verify that for both even and odd $m$ that 
\begin{equation}\label{fmnsneg} f_m^{(n,s)}(t)\le 0, \qquad t\in [-1,s].\end{equation}
 The numbers $\alpha_0,\alpha_1,\ldots,\alpha_{k-1+\varepsilon}$ are the (simple) roots of the equation
\begin{equation}
(t+1)^\varepsilon \left(Q_k^{\varepsilon}(t)Q_{k-1}^{\varepsilon}(\alpha_{k-1+\varepsilon}) -
Q_k^{\varepsilon}(\alpha_{k-1+\varepsilon})Q_{k-1}^{\varepsilon}(t)\right)=0,\label{op_eq}
\end{equation}
where $Q_i^{\varepsilon}(t)=P_i^{(\frac{n-1}{2},\varepsilon+\frac{n-3}{2})}(t)$ are Jacobi polynomials as above, 
$\varepsilon \in \{0,1\}$ (i.e., $(a,b)=(0,\varepsilon)$), and $m=2k-1+\varepsilon$. 
Hereafter we use $\varepsilon \in \{0,1\}$ and $m=2k-1+\varepsilon$ to distinguish between the cases of odd and even $m$.
Note that $s=\alpha_{k-1+\varepsilon}$ and that $-1<\alpha_i < \alpha_{i+1}$ for each $i$ apart from the case $\alpha_0=-1$
which happens if and only if $m=2k$. The Levenshtein quadrature now can be stated as
\begin{equation}
\label{LQF}
f_0= \frac{f(1)}{L_m(n,s)}+ \sum_{i=0}^{k-1+\varepsilon} \rho_i f(\alpha_i),
\end{equation}
holding true for every real polynomial $f(t)$ of degree at most $m=2k-1+\varepsilon$.

Spaces of polynomials where a $1/N$-quadrature rule is valid for some $N$'s (not necessarily integer) and where a solution of
a corresponding linear programming
problem (about lower energy bounds) exists were called ULB-spaces in \cite{BDHSS-hd} (ULB stands for Universal lower bound(s)). 
Theorem \ref{thm construction}
below shows that in the ULB-space $\mathcal{P}_m$ universal upper bounds (UUB) are also featuring.

\subsection{Construction of UUB feasible polynomials}
We use the Levenshtein polynomials to construct feasible polynomials.
Let $n$, $M$, and $s$ be such that the set $C(n,M,s)$ is nonempty (or conjectured to be nonempty). Let $m=m(n,s)$ be as
defined in the previous subsection and let $h \in \mathcal{AM}([-1,1])$.
We consider the polynomial
\begin{equation}
\label{uub_pol}
f(t):=-\lambda f_{m}^{(n,s)}(t)+g_T(t)=\sum_{i=0}^{m} f_i P_i^{(n)}(t),
\end{equation}
where $\lambda>0$ is a parameter (to be determined later) and
\[ g_T(t):=H_{h,T}(t) \]
is the Hermite interpolation polynomial to the function $h(t)$ that agrees with $h(t)$ exactly in
the points of a multiset $T$ (counted with their multiplicities); that is,    $g_T$ interpolates  $h$ and $h'$ agree at repeated nodes.

Note that by the definition of Hermite interpolation the degree of $g_T(t)$ is at most $|T|-1=m-1$.
Thus, $\deg(f)=m$ and this, in particular, implies that we need to verify
non-positivity of the Gegenbauer coefficients $f_i$ for $1 \leq i \leq m$ only. Let
\begin{equation} \label{ellg}f_m^{(n,s)}(t)= \sum_{i=0}^{m} \ell_i P_i^{(n)}(t), \ \ \ g_T(t)=\sum_{i=0}^{\deg(g_T)} g_i P_i^{(n)}(t), \end{equation}
be the Gegenbauer expansions of $f_m^{(n,s)}(t)$ and $g_T(t)$, respectively.
It is important for the applications below that $\ell_i>0$ for every $i=0,1,\ldots,m$ in the Gegenbauer expansion of the
Levenshtein polynomials (see, for example, \cite[Theorem 5.42]{lev98}).

\begin{theorem}\label{thm construction}
Let $n\geq 2$, $M \geq 2$ be an integer, $s \in [-1,1)$, and $h(t)\in \mathcal{AM}([-1,1])$. For any large enough $\lambda>0$, we have that the polynomial $f(t)$ defined as in \eqref{uub_pol} belongs to the class $ U_h^{n,s}$. In particular, if
\begin{equation}
\label{best-alpha}
\lambda=\max \left\{ \frac{g_i}{\ell_i}: 1 \leq i \leq \deg(g_T) \right\},
\end{equation}
the corresponding polynomial $f_m^{(h)}(t) \in U_h^{n,s}$ and
\begin{equation}
\label{uub-main}
\mathcal{G}_h(n,M,s) \leq M\left(\frac{M}{L_m(n,s)}-1\right)f_m^{(h)}(1) +M^2\sum_{i=0}^{k-1+\varepsilon} \rho_i h(\alpha_i).
\end{equation}
A code $C \in C(n,M,s)$ attains the bound \eqref{uub-main} if only if all inner products of $C$ are in $T$ and $(f_m^{(h)})_i M_i(C)=0$ for every $i \geq 1$.
\end{theorem}

\begin{remark} Since the nodes in the upper bound on the energy $\mathcal{G}_h(n,M,s)$ in the right hand side of \eqref{uub-main} do not depend on the potential $h$, we shall refer to this bound hereafter as {\em universal upper bound} or UUB and to the polynomials $f_m^{(h)}(t)$ as {\em UUB feasible polynomials}.
\end{remark}
\begin{proof}
The Hermite interpolation formula with remainder (e.g., see \cite{cdB})  states that there is some $\xi \in [-1,s]$ such that
$$ 
h(t)-g_T(t)=h^{(m)}(\xi)\prod_{\alpha_i \in T} (t-\alpha_i)=h^{(m)}(\xi)f_m^{(n,s)}(t).
$$
Using absolute monotonicity of $h$ and \eqref{fmnsneg} then shows 
$g_T(t) \geq h(t)$ for $t \in [-1,s]$ and hence it follows from \eqref{uub_pol} (again using \eqref{fmnsneg}) that
$f(t) \geq h(t)$ for every $t \in [-1,s]$; i.e. (F1) is satisfied (whatever $\lambda>0$ is).

For (F2), observe that \eqref{uub_pol} implies that the coefficients $f_i$, $i=1,2,\ldots,m$, in the Gegenbauer
expansion of $f(t)$ are in fact the linear combinations
$-\lambda \ell_i + g_i$. Since $\ell_i>0$ for every $i$ it follows that large enough $\lambda>0$ will ensure $f_i \leq 0$ for every $i=1,2,\dots, m$.
Obviously, $f_i=-\lambda \ell_i $ for $\deg (g_T)+1 \leq i\leq m$.
Therefore, (F2) is satisfied (for large enough $\lambda$) and the proof of  $f(t) \in U_h^{n,s}$ is completed.

Since $f_0M-f(1)$ is a linear function of $\lambda$, the conditions $-\lambda \ell_i + g_i \leq 0$
for $i=1,2,\ldots,\deg(g_T)$ imply that the smallest value of $\lambda$ which works is as in \eqref{best-alpha}.

Equality in \eqref{uub-main} implies equality in Theorem \ref{thm 1}, therefore, such equality holds only if all inner products of distinct points in $C$ are in $T$
and $f_iM_i(C)=0$ for every $i \geq 1$.

We now compute the bound produced by $f(t)$. We first note that $L_m(n,s) \geq M$ follows from \eqref{L_bnd}. Indeed, if the converse
$M>L_m(n,s)$ is true, then the monotonicity of the Levenshtien function implies $s(C)>s$, which contradicts to $C \in C(n,M,s)$.

Expressing $f_0$ by the Levenshtein's $1/L_m(n,s)$-quadrature rule we consecutively obtain
\begin{eqnarray*}
\label{uub-weaker}
f_0M-f(1) &=& M\left(\frac{f(1)}{L_m(n,s)}+\sum_{i=0}^{k-1+\varepsilon}\rho_i f(\alpha_i)\right)-f(1) \\
&=& \left(\frac{M}{L_m(n,s)}-1\right)f(1) +M\sum_{i=0}^{k-1+\varepsilon} \rho_i f(\alpha_i) \\
&=& \left(\frac{M}{L_m(n,s)}-1\right)f(1) +M\sum_{i=0}^{k-1+\varepsilon} \rho_i h(\alpha_i)
\end{eqnarray*}
(the last equality follows by using the interpolation conditions $f(\alpha_i)=h(\alpha_i)$, $i=0,1,\ldots,k-1+\varepsilon$) whence we get \eqref{uub-main}.
The dependence of  \eqref{uub-main} on $\lambda$ comes from $f(1)$ only. Since $f(1)$ is linear and
increasing with respect to $\lambda$, the best bound is obtained when $\lambda$ is chosen as in \eqref{best-alpha}.
\end{proof}

\begin{remark} We note that that adding an additional interpolation condition, say by adding a node $-1$ or $s$ to the multiset $T$ in \eqref{MultisetDef},
does not improve the UUB. Indeed, suppose $m=2k-1$ and we consider $T^\prime=\{-1,\alpha_0,\alpha_0,\dots, \alpha_{k-1}\}$.
In this case the interpolation polynomial $g_{T^\prime}(t)=H_{h,T^\prime}(t)$ is of degree $2k-1$ and the interpolation conditions imply that
\[ g_{T^\prime}(t)-g_T(t)=\mu f_{2k-1}^{(n,s)}(t), \]
where $\mu$ is a real number. Then for any polynomial
\[ f(t)=-\lambda^\prime f_{2k-1}^{(n,s)}(t)+g_{T^\prime}(t)=(-\lambda^\prime+\mu) f_{2k-1}^{(n,s)}(t)+g_T(t) \]
and this representation says that, in the optimal case, $-\lambda^\prime+\mu=-\lambda$, where $\lambda$ is chosen as in \eqref{best-alpha}.
Thus we produce the UUB again. The case $m=2k$ can be dealt analogously.
\end{remark}

We next consider the optimality of our bound in a class of feasible polynomials.

\begin{proposition}
Any polynomial $F(t) \in U_h^{n,s}$ of degree at most $m$ satisfying
$F(1) \leq f_m^{(h)}(1)$, where $f_m^{(h)}$ is the polynomial from Theorem \ref{thm construction}, gives an upper bound by Theorem \ref{thm 1} which is not better than
\eqref{uub-main}.
\end{proposition}

\begin{proof}
Assume that $F(t)=\sum_{i=0}^{\deg(F)} F_i P_i^{(n)}(t) \in U_h^{n,s}$, where $\deg(F) \leq m$ and $F(1) \leq f_m^{(h)}(1)$.
As in the proof of Theorem \ref{thm construction}, we see that
\[ F_0M-F(1)=F(1)\left(\frac{M}{L_m(n,s)}-1\right)+M\sum_{i=0}^{k-1+\varepsilon} \rho_i F(\alpha_i). \]
Since $F(1) \leq f_m^{(h)}(1)$, $M \leq L_m(n,s)$, $\rho_i>0$, and $F(\alpha_i) \geq h(\alpha_i)$, we conclude that the inequality
$F_0M-F(1) \geq (f_m^{(h)})_0 \, M-f_m^{(h)}(1)$ follows; i.e., the bound from $F(t)$ is not better than
\eqref{uub-main}.
\end{proof}

\begin{remark}
\label{i=1}
Our numerical experiments suggest that the maximum for the parameter $\lambda$ in Theorem \ref{thm construction}
is always attained at $i=1$. This naturally connects our results to the concept of harmonic index $t$ designs defined
by Bannai, Okuda, and Tagami \cite{BOT}.
\end{remark}

\subsection{Distance distributions of attaining codes}

We consider the combinatorial properties of codes which would attain our UUB.

\begin{definition}
Let $C \subset \mathbb{S}^{n-1}$ be a code with $s(C)=s$. For fixed $x \in C$ and $t \in [-1,s]$, denote by
\[ A_t(x):=|\{ y \in C : \langle x,y \rangle = t \}|. \]
The system of nonnegative integers $\left(A_t(x): t \in [-1,s]\right)$ is called the {\it distance distribution of $C$ with respect to $x$}.
\end{definition}

Assume that $C \subset  \mathbb{S}^{n-1}$ with $s(C)=s$ and $|C|=M$ attains the bound \eqref{uub-main}. Then from Theorem \ref{thm 1}(a), $f_m^{(h)}(t)$ coincides with $h(t)$ on the set $\{\langle x,y\rangle: x\neq y\in C\}$. 
Moreover, from Theorem \ref{thm 1}(b), we have $(f_m^{(h)})_iM_i(C)=0$
for every $i \in \{1,2,\ldots,m\}$.

\begin{theorem}
\label{dd-codes}
In the context of Theorem \ref{thm 1}, if a code $C \subset C(n,M,s)$ attains the bound \eqref{uub-main}
then its distance distribution with respect to $x \in C$ satisfies the system of linear equations
\begin{equation}
\label{system-dd}
1+\sum_{j=0}^{k-1+\varepsilon} A_{\alpha_j}(x) P^{(n)}_i(\alpha_j)=0 \iff M_i(C)=0.
\end{equation}
\end{theorem}

\begin{proof}
The definition \eqref{Mk0} of the moments can be rewritten as
 \begin{equation}\label{Mk-1}
 M_i(C):=M+\sum_{x \in C} \sum_{j=0}^{k-1+\varepsilon} A_{\alpha_j}(x) P^{(n)}_i(\alpha_j).
\end{equation}
Since $M_i(C)=0$ if and only if $\sum_{x \in C} v(x)=0$ for all spherical harmonics $v \in \mbox{Harm}(i)$, we may  use
the addition formula \cite{Koo73} to show that the double sum in \eqref{Mk-1} splits into $|C|$ sums each one equal to $-1$. Indeed, for fixed $y \in C$, we
consecutively obtain
\[ \sum_{x \in C} P_i^{(n)}(\langle x,y \rangle)=\sum_{x \in C} \frac{1}{r_i} \sum_{j=1}^{r_i} v_{ij}(x)\overline{v_{ij}(y)} =
 \frac{1}{r_i}  \sum_{j=1}^{r_i} \overline{v_{ij}(y)} \sum_{x \in C}v_{ij}(x) = 0, \]
where $r_i=\dim \mbox{Harm}(i)$ and $\{v_{ij}(x): j=1,2\ldots,r_i\}$ is an orthonormal basis of Harm$(i)$.
Thus
\[ 1+\sum_{j=0}^{k-1+\varepsilon} A_{\alpha_j}(x) P^{(n)}_i(\alpha_j)=\sum_{y \in C} P_i^{(n)}(\langle x,y \rangle)=0, \]
which completes the proof. \end{proof}

If $M_i(C)=0$ for each $i \in \{2,3,\ldots,m\}$, we obtain $m-1=2k-2+\varepsilon$ linear equations with $k+\varepsilon$ unknowns.
Of course, we add the trivial equation
 \begin{equation}\label{triv-M}
1+\sum_{j=0}^{k-1+\varepsilon} A_{\alpha_j}(x)=M.
\end{equation}

On the other hand, we apply the $1/L_m(n,s)$-quadrature for the polynomials $P_i^{(n)}(t)$, $i=2,3,\ldots,m$ to see that
\[ 1+L_m(n,s)\sum_{j=0}^{k-1+\varepsilon} \rho_j  P^{(n)}_i(\alpha_j)=0. \]
Looking at this as a system with unknowns $L_m(n,s)\rho_j$, $j=0,1,\ldots,k-1+\varepsilon$,
we obtain again \eqref{system-dd} (written for $i=2,3,\ldots,m$). It is easy to see that we have at least as many equations as unknowns for $k \geq 2$.
If the solution is unique, then
\[  A_{\alpha_j}(x) = \rho_j L_m(n,s), \ j=0,1,\ldots,k-1+\varepsilon \]
(in particular, it follows that the distance distribution does not depend on $x$), which leads to $M=L_m(n,s)$ by the trivial equations \eqref{triv-M} and
\[ L_m(n,s)=1+L_m(n,s)\sum_{j=0}^{k-1+\varepsilon} \rho_j \]
(this is the Levenshtein $1/L_m(n,s)$-quadrature for $f(t)=1$).

These observations are summarized in the next theorem.

\begin{theorem}
\label{dd-codes}
In the context of Theorem \ref{thm 1}, if a code $C \in C(n,M,s)$ with $M_i(C)=0$ for each $i \in \{2,3,\ldots,m\}$ attains \eqref{uub-main} then $k=1$ or
the system \eqref{system-dd} has more than one solution.
\end{theorem}

An example of attaining codes with $k=1$ is given in Section 4.1.

\subsection{Test functions}

Next we derive a sufficient condition for optimality of the UUB in
Theorem \ref{thm construction}. For $s \in I_m$ and positive integer $j$, define test functions
\[
R_j^{(n)}(s):= \frac{1}{L_{m}(n,s)}+\sum_{i=0}^{k-1+\varepsilon}\rho_i P_j^{(n)}(\alpha_i),
\]
where the parameters $(\rho_i,\alpha_i)_{i=0}^{k-1+\varepsilon}$ come from the Levenshtein $1/L_m(n,s)$-quadrature.

\begin{theorem}
\label{test-f-th}
In the context of Theorem \ref{thm construction}, if $R_j^{(n)}(s) \geq 0$ for every $j \geq 2k+\varepsilon$, then the bound \eqref{uub-main}
cannot be improved by using a polynomial $F \in U_h^{n,s}$ such that $F(1) \leq f_m^{(h)}(1)$.
\end{theorem}

\begin{proof} Suppose that $R_j^{(n)}(s) \geq 0$ for every positive integer $j$. Write
$F(t) \in U_h^{n,s}$ as
\begin{equation}
\label{n1}
F(t)= u(t)+\sum_{j\geq 2k+\varepsilon}  F_j P_j^{(n)}(t),
\end{equation}
where $u(t)$ is a polynomial of degree at most $2k-1+\varepsilon$ with zeroth Gegenbauer coefficient $u_0$.
It is clear that $F(\alpha_i) \leq h(\alpha_i)$ for $i=0,1,\ldots,k-1+\varepsilon$, $F_j\leq 0$ for every $j\geq 2k+\varepsilon$, and
$F_0=u_0$. Assume also that  $F(1) \leq f_m^{(h)}(1)$.

Using the Levenshtein $1/L_m(n,s)$-quadrature for $u(t)$ and the above relations we consecutively obtain
\begin{eqnarray*}
MF_0- F(1) &=& Mu_0 - u(1)-\sum_{j \geq 2k+\varepsilon} F_j \\
&=& M\left(\frac{u(1)}{L_m(n,s)}+\sum_{i=0}^{k-1+\varepsilon} \rho_i u(\alpha_i)\right)- u(1)-\sum_{j\geq 2k+\varepsilon} F_j \\
    &=& u(1)\left(\frac{M}{L_m(n,s)}-1\right) + M\sum_{i=0}^{k-1+\varepsilon} \rho_i u(\alpha_i)- \sum_{j\geq 2k+\varepsilon}  F_j \\
    &=& \left(F(1)-\sum_{j \geq 2k+\varepsilon} F_j\right) \left(\frac{M}{L_m(n,s)}-1\right) \\
&& \, +
             M\sum_{i=0}^{k-1+\varepsilon} \rho_i \left(F(\alpha_i)-\sum_{j\geq 2k+\varepsilon} F_j P_j^{(n)}(\alpha_i)\right)- \sum_{j\geq 2k+\varepsilon}  F_j \\
   &=& F(1)\left(\frac{M}{L_m(n,s)}-1\right) + M\sum_{i=0}^{k-1+\varepsilon} \rho_i F(\alpha_i)-M \sum_{j\geq 2k+\varepsilon}  F_j R_j^{(n)}(s) \\
  &\geq &  f_m^{(h)}(1)\left(\frac{M}{L_m(n,s)}-1\right) + M\sum_{i=0}^{k-1+\varepsilon} \rho_i h(\alpha_i)=UUB.
\end{eqnarray*}
Hence the bound produced by $F(t)$ is not better than \eqref{uub-main}. \end{proof}

\subsection{The energy strip of $C(n,M,s)$}

Lower bounds for $\mathcal{G}_h(n,M,s)$ can be derived, of course, from constructions of good codes (in a sense of having large energies). 
We present here an analytic approach defining a strip where the energies of all codes from $C(n,M,s)$ lie and, in particular, a lower bound on
$\mathcal{G}_h(n,M,s)$. More precisely, we combine the upper bound from Theorem \ref{thm construction} and the
universal lower bound from \cite{BDHSS-ca} to obtain a strip where all possible energies of codes from $C(n,M,s)$ belong.

To explain the lower bounds we start with setting with $M=L_{m}(n,r)$ for a unique $r \in I_m$, where $m=2k-1+\varepsilon$, $\varepsilon \in \{0,1\}$, as above
(the uniqueness follows from the strict monotonicity of the Levenshtein bounds). Note that $r \leq s$ as equality holds if and only if there 
exists a universally optimal code of cardinality $M=L_m(n,s)$ (see also the comment after the next theorem). 
Let the Levenshtein polynomial $f_{m}^{(n,r)}(t)$ have roots $\alpha_0^\prime<\alpha_1^\prime<\cdots<\alpha_{k-1+\varepsilon}^\prime=r$ 
with corresponding weights
$\rho_0^\prime, \rho_1^\prime,\ldots,\rho_{k-1+\varepsilon}^\prime$ in the Levenshtein $1/M$-quadrature rule.

\begin{theorem}
\label{strip_thm}
The energy of any code $C \in C(n,M,s)$ is bounded from below and above by
\begin{equation}
\label{ulb-uub-odd}
M^2 \sum_{i=0}^{k-1+\varepsilon} \rho_i^\prime h(\alpha_i^\prime) \leq E_h(C) \leq
M\left(\frac{M}{L_{m}(n,s)}-1\right)f(1) +M^2\sum_{i=0}^{k-1+\varepsilon} \rho_i h(\alpha_i) ,
\end{equation}
\end{theorem}

It is clear from the above that $M=L_m(n,s)$ implies the coincidence of the upper and lower bounds in \eqref{ulb-uub-odd}. In this case
the corresponding codes are sharp configurations (also universally optimal codes; see \cite{BOT,CK07,lev98}) which means that they attain simultaneously
the Levenshtein bound, the ULB \cite{BDHSS-ca} and the UUB from Theorem \ref{thm construction} (as the last two coincide and the
strip \eqref{ulb-uub-odd} becomes a point). Two prominent examples are given by the simplex code $C \in C(n,n+1,-1/n)$ and the 
cross polytope (also known as bi-orthogonal code) $C \in C(n,2n,0)$.

\section{Examples }

\subsection{Orthonormal basis codes and the UUB}
First, we now provide an example of a code where the UUB is attained. Moreover, we obtain more than one optimal polynomials.


Suppose $C \subset \mathbb{S}^{n-1}$ consists of orthonormal basis vectors in $\mathbb{R}^n$ and select $s=0$. Then $M=n$ and
\[ E_h(C)=n(n-1)h(0). \]
Clearly, the constant polynomial $f(t):=h(0)$ provides one solution to the LP problem. We now determine a second one.

In our construction, we have $m=2$,
\[ f_2^{(n,0)}(t)=t(t+1)=\frac{1}{n}P_0^{(n)}(t)+P_1(t)+\frac{n-1}{n}P_2^{(n)}(t), \]
$T=L=\{-1,0\}$, and
\[ g_T(t)=(h(0)-h(-1))t+h(0)=h(0)P_0^{(n)}(t)+(h(0)-h(-1))P_1^{(n)}(t). \]
The choice $\lambda=h(0)-h(-1)>0$ yields
\begin{eqnarray*}
f(t) &=& h(0)-(h(0)-h(-1))t^2 \\
      &=& \frac{(n-1)h(0)+h(-1)}{n}P_0^{(n)}(t)-\frac{(n-1)(h(0)-h(-1))}{n}P_2^{(n)}(t) \in U_{n,0}^h .
\end{eqnarray*}
Our bound can be computed directly by
\[ n(nf_0-f(1))=n(n-1)f(0)=n(n-1)h(0)=E_h(C). \]
Since $k=1$, $\varepsilon=1$, we find the quadrature nodes and weights to be $(\alpha_0,\rho_0)=(-1,1/2n)$, $(\alpha_1,\rho_1)=(0,(n-1)/n)$. Computing the Levenshtein function $L_2 (n,0)=2n$ (note that
$L_3(n,0)=2n$ as well), the right-hand side of \eqref{uub-main} becomes
\[ n\left(\frac{n}{2n}-1\right)h(-1) +n^2\left(\frac{h(-1)}{2n}+\frac{(n-1)h(0)}{n}\right)=n(n-1)h(0)=E_h(C).  \]

\subsection{Bounds for $(n,M)=(n,2n+1)$ codes}

It is natural to consider upper bounds for parameters where good codes are known. Here we show how our bound behaves
for spherical codes $C_n \subset \mathbb{S}^{n-1}$ with $M=2n+1$ points constructed in \cite{EZ01}. These codes are conjectured
to be optimal (see \cite[Section 3.3]{BBCGKS}) but this is proved in dimensions 3 \cite{SvdW} and 4 \cite{ZZ17} only.

The maximal inner product of $C_n$ is equal to the unique root $s \in (0,1/n)$ of the equation
\[ n(n-2)^2 X^3-n^2X^2-nX+1=0. \]
These parameters are in the region of the third Levenshtein bound; i.e., we use $m=3$.

The ULB \cite{BDHSS-ca} with parameters coming from $L_3(n,r)=M$ as in Section 3.5 is
\begin{equation}
\label{ulbn5M11} R_h(n,2n+1):=M^2\left(\rho_0^\prime h(\alpha_0^\prime)+\rho_1^\prime h(r)\right), \ r=\alpha_1^\prime.
\end{equation}

To obtain the UUB in this case we consider the corresponding Levenshtein polynomial $f_3^{(n,s)}(t)$ with zeros $\alpha_0$ (double) and
$\alpha_1=s$ (simple). Then $T=\{\alpha_0,\alpha_0,\alpha_1\}$ and
$g(t):=H_{h,T}(t)$
is the second degree interpolant to $h$ in the nodes $\alpha_0$ (doubly) and $\alpha_1$. The polynomial from \eqref{uub_pol} is
\[ f(t)=-\lambda  f_3^{(n,s)}(t)+g(t)=\sum_{i=0}^3 f_i P_i^{(n)}(t) \]
where $\lambda>0$ has to be chosen to ensure $f_1 \leq 0$ and $f_2 \leq 0$ ($f_3<0$ follows for every $\lambda>0$).

Here are the numerical results for $n=5$, $M=2n+1=11$ and $s \approx 0.13285$ with the Newton potential $h(t)=1/(2-2t)^{(n-2)/2}$.

The lower bound from \eqref{ulb-uub-odd} (specified in \eqref{ulbn5M11}) is
\[ R_h(5,11)=11^2(\rho_0^\prime h(\alpha_0^\prime)+\rho_1^\prime h(\alpha_1^\prime)) \approx 37.484. \]

For the construction of the upper bound we find $f_3^{(5,s)}(t)$ with roots $(\alpha_0,\alpha_1=s) \approx (-0.68069,0.13285)$
and
\[ g(t)=H(h;\alpha_0,\alpha_0,\alpha_1) \approx 0.23835 t^2+0.46931 t + 0.37128. \]
Then we search for $\lambda$ to satisfy the conditions $f_i \leq 0$ for $i=1,2,3$.
The computations show that all $\lambda \geq 0.661$ work as best the upper bound $\approx 41.906$ from Theorem \ref{thm construction}
(also the upper bound in \eqref{ulb-uub-odd}) is obtained with the smallest possible $\lambda=g_1/\ell_1 \approx 0.661$.
For the representation of the upper bound in \eqref{ulb-uub-odd} we  compute $L_3(5,s) \approx 13.3014$ and
\begin{equation}
\label{weak-uub-3}
\mathcal{G}_h(n,M,s) \leq 11\left(\frac{11}{L_3(5,s)}-1\right)f(1)+11^2\left(\rho_0 h(\alpha_0)+\rho_1 h(s)\right).
\end{equation}

The Newton energy of the code $C_5$ is
\[ E_h(C_5)=(3n^2-n)h(s)+(n^2-n)h(a)+2nh(b)+2nh(c) \approx 39.0225, \]
where $a \approx -0.22793$, $b \approx -0.553428$, and $c \approx -0.89904$. The best known (for the
minimum Newton energy problem) code of 11 points on $\mathbb{S}^4$ has energy $\approx 38.0544$ \cite{BBCGKS}.

\begin{table}[h]
\caption{ULB and UUB for Newton energy of kissing numbers in dimensions $2-10$}
\label{KN}

\vspace*{3mm}

\centering
\begin{tabular}{|c|c|c|c|c|c|c|}
\hline
 $n$  &  Kissing numbers & $m$ &  $L_m(n,1/2)$ &      ULB  &  UUB  \\ 
  
   &  bounds \cite{EZ01,BV08,MV10,MOF18} & &   &  & \\ \hline
   $2$ &  $6$ & $5$  & 6  & -10.75...  &   -10.75...   \\ \hline     
     $3$ &  $12$ & $5$  &  13.2  & 98.3  &   101.3  \\ \hline     
  $4$ &  $24$   &5 &  $26$  &   $333$  & $344$         \\ \hline
  $5$ &  $40,\ldots,44$   &$6$ &  $48$  &    $765.,\ldots,947.$ &  $840.,\ldots,989.$  \\ \hline
  $6$ &  $72,\ldots,78$   &$6$ &  $84.$ &    $2116.,\ldots,2530.$  & $2218.,\ldots,2594.$     \\ \hline       
 $7$ &  $126,\ldots,134$ &  $6$  & $142.$    & $5552.,\ldots,6376. $  & $5793.,\ldots,6514.$    \\ \hline        
  $8$ &  $240$ & $7$  & 240   & 17721.  &   17721.   \\ \hline      
  $9$ &  $306,\ldots,363$ &  $7$  & $384.$    & $23149.,\ldots,34231.$  & $27443.,\ldots,35616.$    \\ \hline    
    $10$ &  $500,\ldots,554$ &  $7$  & $605$    & $53059.,\ldots,67004.$  & $61467.,\ldots,71606.$    \\ \hline      
 \end{tabular}
\end{table}

\subsection{Upper and lower bounds on energy of kissing configurations}

We now consider upper bounds on the Newtonian energies of {\em kissing configurations}. For such configurations $s=1/2$. Depending on the dimension $n$ we have various known bounds in the literature for the corresponding {\em kissing number}, i.e. the maximum possible cardinality of a code that has a separation parameter $s=1/2$.  The exact kissing numbers are known for dimensions $n=1,2,3,4,8,24$ (see \cite{lev79, M, OS79,SvdW52}).

Table \ref{KN} summarizes our results for $n=2-10$. We list the dimension, known kissing number intervals, Levenshtein interval $m$ and Levenshtein function value $L_m(n,1/2)$ and the corresponding ULB and UUB intervals. Numbers are rounded to integer parts, rounding being indicated with the decimal point. 
\vskip 3mm

\subsection{Conclusion and future work}

The conditions for attaining the bound of Theorem \ref{thm construction} lead to the usual suspects -- the universally optimal
configurations defined in \cite{CK07}. From a broader viewpoint, our upper bounds help provide a range of possible energies (or energy levels) for `good' spherical codes. Thereby, we obtain restrictions on the structure of codes that can be useful for classification (or nonexistence) purposes.  We plan to explore this idea in a future paper that relates to kissing configurations.


Several additional related questions arise quite naturally. For example, whether the optimality condition $f_1=0$ (see Remark \ref{i=1}) is true for every absolutely monotone potential function $h$. A second question is whether effective `next level' upper bounds can be developed
in a manner similar to that derived for lower bounds by the present authors in \cite{BDHSS-hd} (see also \cite{CK07}). 

{\bf Acknowledgments.} The authors thank the anonymous referees for helpful suggestions and comments.




\end{document}